\documentclass[11pt, a4paper]{article}
\usepackage{geometry}
\usepackage{xcolor}
\usepackage{float}
\usepackage{textcomp}
\usepackage{url}
\usepackage{amsthm}
\usepackage{amsmath}
\usepackage{amssymb}%
\usepackage{mdframed}
\usepackage{multirow}
\usepackage{changes}
\usepackage{hyperref}

\usepackage{mathtools}
\usepackage{booktabs}    
\usepackage{subfigure} 
\usepackage{caption} 
\usepackage{multirow}
\usepackage{multicol}    
\usepackage{array}       
\usepackage{graphicx}
\graphicspath{{figuras/}}
\usepackage{empheq,amsmath}

\usepackage
[ 
lined, 
boxruled, 
commentsnumbered
]
{algorithm2e}
\DecMargin{0.1cm} 

\DeclarePairedDelimiter{\ceil}{\lceil}{\rceil}
\newtheorem{theorem}{Theorem}[section]
\newtheorem{lemma}[theorem]{Lemma}

\newtheorem{remark}[theorem]{Remark}

\DeclareMathOperator{\argmin}{argmin}

\newcommand{\N}{\mathcal{N}}
\newcommand{\R}{\mathbb{R}}

\newcommand{\G}{\mathcal{G}}
\newcommand{\D}{\mathcal{D}}

\newcommand{\cL}{{\mathcal{L}}}

\newcommand{\bi}{\begin{itemize}}
\newcommand{\ei}{\end{itemize}}
\newcommand{\ba}{\begin{array}}
\newcommand{\ea}{\end{array}}

\usepackage{todonotes}
\usepackage{xcolor}

\begin{document}

\title{\textbf{Sub-sampled Trust-Region Methods with Deterministic Worst-Case Complexity Guarantees}}


 \author{
 Max L. N. Gon\c calves\thanks{Instituto de Matem\'atica e Estat\'istica, Universidade Federal de Goi\'as,  Goi\^ania-GO, 74001-970,  Brazil, E-mail: \texttt{maxlng@ufg.br}. This author was partially supported by FAPEG (Grant 405349/2021-1) and CNPq (Grant 304133/2021-3).} \and  Geovani N. Grapiglia\thanks{Universit\'e catholique de Louvain, ICTEAM/INMA, Avenue Georges Lema\^itre, 4-6/ L4.05.01, B-1348, Louvain-la-Neuve, Belgium. E-mail: \texttt{geovani.grapiglia@uclouvain.be}. This author was partially supported by FRS-FNRS, Belgium (Grant CDR J.0081.23).} 
 }

\maketitle

\begin{abstract}
In this paper, we develop and analyze sub-sampled trust-region methods for solving finite-sum optimization problems. These methods employ subsampling strategies to approximate the gradient and Hessian of the objective function, significantly reducing the overall computational cost. We propose a novel adaptive procedure for deterministically adjusting the sample size used for gradient (or gradient and Hessian) approximations. Furthermore, we establish worst-case iteration complexity bounds for obtaining approximate stationary points. More specifically, for a given $\varepsilon_g, \varepsilon_H\in (0,1)$, it is  shown that an $\varepsilon_g$-approximate first-order stationary point  is reached  in at most $\mathcal{O}({\varepsilon_g}^{-2} )$  iterations, whereas an  $(\varepsilon_g,\varepsilon_H)$-approximate second-order stationary point  is reached  in at most $\mathcal{O}(\max\{\varepsilon_{g}^{-2}\varepsilon_{H}^{-1},\varepsilon_{H}^{-3}\})$ iterations. Finally, numerical experiments illustrate the effectiveness of our new subsampling technique.
\\[2mm]
\end{abstract}

\section{Introduction}\label{sec:1}

\noindent\textbf{Motivation and Contributions.} We consider the finite-sum minimization problem  
\begin{equation}
\min_{x\in\mathbb{R}^{n}}\,f(x):=\frac{1}{d}\sum_{i=1}^{d}f_{i}(x), \label{prob1}
\end{equation}
where each $f_{i}:\mathbb{R}^{n}\to\mathbb{R}$ is twice continuously differentiable, but possibly nonconvex.  
Problems of the form~\eqref{prob1} lie at the heart of data fitting applications, where the decision variable~$x$ typically represents the parameters of a model, and each function~$f_i(x)$ measures the discrepancy between the model's prediction and the $i$th data point. In this context, (\ref{prob1}) corresponds to minimizing the average prediction error over $d$ data points. 

When the number $d$ of component functions in (\ref{prob1}) is large, the computation of 
\begin{equation*}
\nabla f(x)=\frac{1}{d}\sum_{i=1}^{d}\nabla f_{i}(x)\quad\text{and}\quad \nabla^{2}f(x)=\frac{1}{d}\sum_{i=1}^{d}\nabla^{2}f_i(x),    
\end{equation*}
may become excessively expensive, which severely impacts the performance of first- and second-order methods applied to solve~\eqref{prob1}. To mitigate this issue, several \textit{sub-sampled optimization methods} have been proposed in recent years (e.g., \cite{ BEL5,BEL6,10.1093,BEl3,BEL2,JMLR:v23:20-910,Curtis,serafino,Max2024,OJMO_2025__6__A5_0,11850,Pardalos1,Bell3,Lan1,Roosta2020,Roosta2021}). These methods rely on inexact derivative information computed via subsampling. Specifically, given a point~$x$, they approximate the gradient and/or Hessian as
\begin{equation}
\nabla f_{\mathcal{G}} (x) = \frac{1}{|\mathcal{G}|} \sum_{i \in \mathcal{G}} \nabla f_i(x), \quad \text{and} \quad
\nabla^2 f_{\mathcal{H}}(x) = \frac{1}{|\mathcal{H}|} \sum_{i \in \mathcal{H}} \nabla^2 f_i(x), 
\label{subsamGradi}
\end{equation}
where \( \mathcal{G}, \mathcal{H} \subset \{1, \ldots, d\} \) are sub-samples of the data indices, and \( |\mathcal{G}| \) and \( |\mathcal{H}| \) denote their respective cardinalities, with
\begin{equation*}
    f_{\mathcal{G}}(x)=\dfrac{1}{|\mathcal{G}|}\sum_{i\in \mathcal{G}}\,f_{i}(x).
\end{equation*}
Usually, the samples are chosen randomly, with either adaptive or predefined control over their cardinality. With randomized samples, worst-case complexity bounds are typically established in expectation or with high probability (e.g., \cite{BEL5,BEL6,Cartis,JMLR:v23:20-910,Curtis,serafino,11850,Pardalos1,Lan1,Roosta2021}). On the other hand, with a predefined schedule for the sample sizes, one can recover deterministic complexity guarantees, provided that the full sample size is eventually reached (e.g., \cite{BEl3,BEL2,Max2024,Bell3}).

In the present work, we explore a different avenue based on \textit{deterministic error bounds} for sub-sampled gradient and Hessian approximations, where the accuracy is determined by the cardinality of the samples. Leveraging these error bounds, we develop and analyze sub-sampled trust-region methods for solving finite-sum optimization problems, with exact function evaluations. The sample sizes are selected \textit{deterministically} and in a \textit{fully adaptive} manner. We establish worst-case iteration complexity bounds for obtaining approximate first- and second-order stationary points. Specifically, for given tolerances $\varepsilon_g, \varepsilon_H \in (0,1)$, we show that our first-order trust-region method requires at most $\mathcal{O}(\varepsilon_g^{-2})$ iterations to find a point $\bar{x}$ such that
\[
\|\nabla f(\bar{x})\| \leq \varepsilon_g,
\]
assuming Lipschitz continuity of the gradients. In addition, assuming also Lipschitz continuity of Hessians, we show that our second-order trust-region method requires at most $\mathcal{O}(\max\{\varepsilon_g^{-2} \varepsilon_H^{-1}, \varepsilon_H^{-3}\})$ iterations to find a point $\bar{x}$ satisfying
\[
\|\nabla f(\bar{x})\| \leq \varepsilon_g \quad \text{and} \quad \lambda_{\min}(\nabla^2 f(\bar{x})) \geq -\varepsilon_H.
\]
Finally, we present numerical results that illustrate the potential benefits of our new subsampling strategy.
\\[0.3cm]
\noindent\textbf{Contents.} The paper is organized as follows. Section~\ref{sec:2}  describes  the main assumptions made  throughout this work and establishes  crucial auxiliary results.  Section \ref{sec:3}  presents  and analyzes a sub-sampled trust region method for obtaining approximate first-order stationary  points of $f(\,\cdot\,)$, whereas Section~\ref{sec:4} is devoted to present an extension  of this first algorithm  for obtaining approximate second-order stationary  points. Section~\ref{NumSec} presents preliminary numerical experiments and some concluding remarks are given in Section \ref{concluding}
\\[0.3cm]
\noindent\textbf{Notation.} The symbol $\|\,\cdot\,\|$ denotes the $2$-norm for vectors or matrices (depending on the context). The Euclidian inner product of $x,y\in\mathbb{R}^{n}$ is denoted by $\langle x,y\rangle$. For a given $z\in \R_+$, we set $\ceil{z}:=\min\{x\in \mathbb{Z}_{++}: x\geq z\}$.  Furthermore, the identity matrix of $\mathbb{R}^{n\times n}$ is denoted by $I$, and for any symmetric matrix $A\in\mathbb{R}^{n\times n}$, $\lambda_{\min}\left(A\right)$ is the minimum eigenvalue of $A$. We denote $\mathcal{N}=\left\{1,\ldots,d\right\}$, and for a given subsample  $\mathcal{G}\subset \N$, {$\mathcal{N}\setminus\mathcal{G}$ denotes the set $\{i\in\left\{1,\ldots,d\right\}: i \notin \mathcal{G}\}$}.

\section{Assumptions and Auxiliary Results}\label{sec:2}
This subsection presents the main assumptions made  throughout this work and establishes  crucial auxiliary results. We begin by proving a technique  result.
\begin{lemma}
\label{lem:1}
 Given $x\in\mathbb{R}^{n}$ and $h\in [0,1]$,  let $\mathcal{G} \subset \N$ be  such that  $|\mathcal{G}|\geq\ceil{(1-h)|\N|}$. Then, 
 \begin{equation}
\|\nabla f(x)-\nabla f_{\mathcal{G}}(x)\|\leq 2h \max_{i \in \N} \|\nabla f_i(x)\|, \quad{\text{and}}\quad \|\nabla^2 f(x)-\nabla^2 f_{\mathcal{G}}(x)\|\leq 2h \max_{i \in \N} \|\nabla^2 f_i(x)\|.
\label{eq:2}
\end{equation}
\end{lemma}
\begin{proof}
It follows from \eqref{prob1} and \eqref{subsamGradi} that
\begin{align*}
  \nabla f(x)-\nabla f_{\mathcal{G}}(x)&= \frac{1}{|\N|}\sum_{i \in \N} \nabla f_i(x)-\frac{1}{|\mathcal{G}|}\sum_{i\in \mathcal{G}} \nabla f_i(x) \\ 
  &= \frac{1}{|\N||\mathcal{G}|}\left[|\mathcal{G}|\sum_{i\in \N} \nabla f_i(x)-|\N|\sum_{i\in \mathcal{G}} \nabla f_i(x)\right] \\
  &= \frac{1}{|\N||\mathcal{G}|}\left[(|\mathcal{G}|-|\N|)\sum_{i \in \mathcal{G}} \nabla f_i(x)+|\mathcal{G}|\sum_{i\in{\mathcal{N}\setminus\mathcal{G}}} \nabla f_i(x)\right], 
\end{align*}
which, combined with the fact that $|\N|\geq |\mathcal{G}|$, yields   
\begin{align*}
  \|\nabla f(x)-\nabla f_{\mathcal{G}}(x)\|&\leq  
   \frac{|\N|-|\mathcal{G}|}{|\N||\mathcal{G}|}\sum_{i \in \mathcal{G}} \|\nabla f_i(x)\|+\frac{1}{|\N|}\sum_{i\in {\mathcal{N}\setminus\mathcal{G}}}\| \nabla f_i(x)\|\\ 
   &{\leq \left(\frac{|\N|-|\mathcal{G}|}{|\N||\mathcal{G}|}\right)|\mathcal{G}|\max_{i\in\mathcal{G}}\|\nabla f_{i}(x)\|+\frac{|\mathcal{N}\setminus\mathcal{G}|}{|\mathcal{N}|}\max_{i\in\mathcal{N}\setminus\mathcal{G}}\|\nabla f(x)\|}\\
   &=   \frac{|\N|-|\mathcal{G}|}{|\N|} \max_{i \in \mathcal{G}} \|\nabla f_i(x)\|+\frac{|\N|-|\mathcal{G}|}{|\N|}\max_{i \in{\mathcal{N}\setminus\mathcal{G}}} \|\nabla f_i(x)\|\\
  & \leq 2  \left(\frac{|\N|-|\mathcal{G}|}{|\N|}\right) \max_{i \in \N} \|\nabla f_i(x)\|.
\end{align*}
Therefore, the first inequality in \eqref{eq:2} now follows from the fact that $|\mathcal{G}|\geq {(1-h)|\N|}$. The proof of the second inequality in \eqref{eq:2} follows a similar structure to that of the first and is therefore omitted for brevity.
\end{proof}

The following assumption is made throughout this work:
\begin{itemize}
\item[\textbf{(A1)}] {For every $\mathcal{G}\subset\mathcal{N}$, $f_{\mathcal{G}}$ is twice-continuously differentiable}, and the gradient  $\nabla f_{\G}$ is $L_g$-Lipschitz continuous, i.e.,
\begin{equation} \label{eq:o09}
\|\nabla f_{\G}(y)-\nabla f_{\G}(x)\|\leq L_g\|y-x\|,\quad \forall x,y\in\mathbb{R}^{n}.
\end{equation}
\item[\textbf{(A2)}]  {For every $i\in\N$, there exists $x^*_ i$ such that  $\nabla f_i(x^*_i)=0$. In addition, there exists $f_{low}\in\mathbb{R}$ such that}
\begin{equation}\label{inf:f}
   f(x) \geq f_{\text{low}}, \quad \forall x \in \mathbb{R}^{n}.
\end{equation}
\item [\textbf{(A3)}] {For every  $x_0 \in \R^n$, the sublevel set 
\begin{equation*}
    \mathcal{L}_{f}(x_{0}):=\left\{x\in\mathbb{R}^{n}\,:\,f(x)\leq f(x_{0})\right\}
\end{equation*}
is bounded.}

\end{itemize}

{
\begin{remark}
In view of {\bf (A3)}, given $x_{0}$ we have
\begin{equation}\label{def:D0}
\D_0:=\sup_{x \in \cL_f(x_0)}\max_{i\in \N}\|x-x^*_ i\| <\infty,
\end{equation}
where $x_{i}^{*}$ ($i=1,\ldots,d$) are the points specified in {\bf (A2)}.
\end{remark}
}

\begin{remark}  We note that, from {\bf (A1)}, it can be shown that, 
\begin{equation}
f(y)\leq f(x)+\langle\nabla f(x),y-x\rangle+\dfrac{L_g}{2}\|y-x\|^{2}, \quad \forall x,y\in\mathbb{R}^{n},
\label{eq:2.2}
\end{equation}
and 
 \begin{equation}\label{Li:Hessi}
   \|\nabla^2 f_{\mathcal{H}}(x)\|\leq L_g, \quad \forall x \in \R^n, \mathcal{H}\subset \N.    
 \end{equation}
\end{remark}

The next lemma, in particular, establishes  error bounds when the gradient and  Hessian of  the objective function $f$ are computed inexactly using subsampling techniques.

\begin{lemma} \label{lem:2}
Suppose that {\bf (A1)--(A3)} hold. Let $\varepsilon_g, \varepsilon_H>0$ and $x\in \cL_f(x_0)$. 
For a  given   $h\in [0,1]$, define   $\G, \mathcal{H} \subset \N$   such that  $|\mathcal{G}|\geq\ceil{(1-h)|\N|}$ and $|\mathcal{H}|\geq\ceil{(1-h)|\N|}$. Then,
  \begin{itemize}
 \item[(a)]
 $\|\nabla f(x)-\nabla f_{\mathcal{G}}(x)\|\leq 2hL_g \D_0;$ 
  \item[(b)] $-\lambda_{\min}(\nabla^2 f(x)) \leq -\lambda_{\min}(\nabla^2 f_{\mathcal{H}}(x)) +2hL_g$;
 \item[(c)] {If 
 \begin{equation*}
     \|\nabla f(x)\|>\varepsilon_{g}\quad\text{and}\quad h<\dfrac{\varepsilon_{g}}{10L_{g}D_{0}},
 \end{equation*}
 then $\|\nabla f_{\mathcal{G}}(x)\|>4\varepsilon_{g}/5$};
 
  \item[(d)] {If 
  \begin{equation*}
      -\lambda_{\min}\left(\nabla^{2}f(x)\right)>\epsilon_{H}\quad\text{and}\quad h<\dfrac{\epsilon_{H}}{10 L_{g}},
  \end{equation*}
  then $-\lambda_{\min}\left(\nabla^{2} f_{\mathcal{H}}(x)\right)>4\epsilon_{H}/5$}.
 \end{itemize}
\end{lemma}
\begin{proof}
(a) It follows from the first inequality in \eqref{eq:2}, the fact that $\nabla f_i(x^*_i)=0$ for every $i \in \N$,  and inequalities in \eqref {eq:o09}  and \eqref{def:D0} that 
 \[
 \|\nabla f(x)-\nabla f_{\mathcal{G}}(x)\|\leq 2h \max_{i \in \N} \|\nabla f_i(x)-\nabla f_i(x^*_i)\|\leq 2h L_g \max_{i \in \N}  \|x-x^*_i\|\leq 2hL_g\D_0,
 \]
 which proves the desired inequality.
\\[2mm]
(b) From the second inequality in \eqref{eq:2} with $\mathcal{G}=\mathcal{H}$ and \eqref{Li:Hessi}, we  find that
 \begin{equation}\label{er:908}
  \|\nabla^2 f(x)-\nabla^2 f_{\mathcal{H}}(x)\|\leq 2h \max_{i \in \N} \|\nabla^2 f_i(x)\|\leq 2h L_g.    
 \end{equation}
Hence, for any $d\in\mathbb{R}^{n}$, it follows that
\begin{align*}
{\langle\left(\nabla^2 f_{\mathcal{H}}(x)-\nabla^{2}f(x)\right)d,d\rangle}
\leq \|\nabla^2 f(x)-\nabla^{2}f_{\mathcal{H}}(x)\|\|d\|^{2}
 \leq  2hL_g\langle I\,d,d\rangle.
\end{align*}
Since the inequality above holds for all $d\in\mathbb{R}^{n}$, we have
\begin{equation*}
\nabla^{2}f_{\mathcal{H}}(x)\preceq\nabla^2 f (x)+ 2hL_gI,
\end{equation*}
which, using the Weyl's inequality \cite{doi:10.1137/1.9781611971446}, yields
\[
-\lambda_{\min}\left(\nabla^2 f(x)\right)\leq-\lambda_{\min}\left(\nabla^{2}f_{\mathcal{H}}(x)\right)+2hL_g,
\] 
concluding the proof of the item.
\\[2mm]
 (c) By combining $\|\nabla f(x)\|>\varepsilon_g$,  the inequality of   item~(a) and  $h<\varepsilon_g/(10L_g\D_0)$, we have  
 \[
 \varepsilon_g<\|\nabla f(x)\|\leq \|\nabla f(x)-\nabla f_{\mathcal{G}}(x)\|+\|\nabla f_{\G}(x)\|\leq 2hL_g\D_0+\|\nabla f_{\G}(x)\|\leq \frac{\varepsilon_g}{5} +\|\nabla f_{\G}(x)\|,
 \]
 which implies the inequality of item~(c).
 \\[2mm]
 (d) It follows from   $-\lambda_{\min}(\nabla^2 f(x_k))>\varepsilon_H$, the inequality of   item~(b) and   $h<\varepsilon_H/(10L_g)$ that   
 \[
 \varepsilon_H<-\lambda_{\min}(\nabla^2 f(x)) \leq -\lambda_{\min}(\nabla^2 f_{\mathcal{H}}(x)) +2hL_g\leq -\lambda_{\min}(\nabla^2 f_{\mathcal{H}}(x))+\frac{\varepsilon_H}{5},
 \]
 which implies the desired inequality of  the item.
\end{proof}

\section{Method for Computing Approximate First-Order Critical Points}\label{sec:3}

In this section, we present and analyze a sub-sampled trust-region method for finding approximate first-order stationary points of problem~\eqref{prob1}. The method employs an adaptive subsampling strategy to approximate the gradient of \( f \), while the Hessian is approximated by a symmetric matrix, which may or may not be computed via subsampling. Specifically, our novel adaptive sampling procedure selects a subset \( \mathcal{G}_k \subset \mathcal{N} \) such that
\begin{equation}
\|\nabla f(x_k) - \nabla f_{\mathcal{G}_k}(x_k)\| \leq \mathcal{O}(\Delta_k) \quad \text{and} \quad \|\nabla f_{\mathcal{G}_k}(x_k)\| > \frac{4\varepsilon_g}{5}, 
\label{eq:keys}
\end{equation}
where \( \varepsilon_g > 0 \) is a user-defined tolerance for the norm of the gradient of \( f \). Thanks to the inequalities in~\eqref{eq:keys}, we show that Algorithm~1 requires at most \( \mathcal{O}(\varepsilon_g^{-2}) \) iterations to compute an \( \varepsilon_g \)-approximate stationary point of the objective function.
\newpage
\begin{mdframed}
\noindent\textbf{Algorithm 1:} {First-Order sub-sampled-TR Method}
\\[0.2cm]
\noindent\textbf{Step 0.} {Given} $x_0 \in \R^n$, $\varepsilon_g > 0$, $\gamma>1,$ $\alpha \in (0, 1)$ and  $\Delta_{\max}\geq\Delta_{0}>0$, set $k:=0$.
\\[0.2cm]
\noindent\textbf{Step 1.} Let $j:=0$.
\\[0.2cm]
\noindent\textbf{Step 1.1.} Define 
\begin{equation} \label{def:h_k}
   {h^j_k:=\frac{\Delta_k}{\gamma^j\Delta_{\max}}}, 
\end{equation}
and choose {$\G_k^j\subset \N$ such that $|\mathcal{G}_k^j|\geq \ceil{(1-h^j_k)|\N|}$}. 
\\[0.2cm]
\noindent\textbf{Step 1.2.} Compute $\nabla f_{\mathcal{G}_k^j}(x_k)$. If 
\begin{equation}\label{graSubcond}
\|\nabla f_{\mathcal{G}_k^j}(x_k)\|> \frac{4\varepsilon_g}5,   
\end{equation}
then {define} $j_k=j$ and $\G_k:=\mathcal{G}_k^{j_k}$, and go to Step~2. Otherwise, set $j:=j+1$ and go  to Step 1.1. 
\\[0.2cm]
\noindent\textbf{Step 2.} {Compute $B_{k}\in\mathbb{R}^{n\times n}$ symmetric.}
\\[0.2cm]
\noindent\textbf{Step 3} {Compute an approximate solution $d_{k}$ of the trust-region subproblem}  
\begin{equation}
\begin{aligned} 
\min_{\|d\| \leq \Delta_{k} } \quad & m_k(d):= f(x_k)+\langle \nabla f_{\mathcal{G}_k}(x_k),d\rangle+\dfrac{1}{2}\langle B_{k}d,d\rangle,\\
\end{aligned}
\label{eq:subproblem}
\end{equation}
 such that
\begin{equation}\label{sdcdf}
   m_k(0)-m_k(d_k) \geq \frac12\|\nabla f_{\mathcal{G}_k}(x_k)\|\min\left\{\Delta_k,\frac{\|\nabla f_{\mathcal{G}_k}(x_k)\|}{{\|B_{k}\|}}\right\}.
    \end{equation}
\noindent\textbf{Step 4.} Compute \begin{equation}\label{ratio_generalset}
    \rho_{k}=\frac{f(x_k)-f(x_k+d_{k})}{m(0)-m(d_{k})}.
\end{equation} 
If $\rho_{k}\geq\alpha$, define $x_{k+1}:=x_{k}+d_{k}$ and $\Delta_{k+1}:=\min\left\{2\Delta_{k},\Delta_{\max}\right\}$. Otherwise, set $x_{k+1}:=x_{k}$ and  $\Delta_{k+1}:=\frac{1}{2}\Delta_{k}$.
\\[0.2cm]
\noindent\textbf{Step 5} Set  $k:=k+1$ and return to Step~1.
\end{mdframed}

\vspace{0.3cm}
\begin{remark}
(i) 
As will be shown, the well-definedness of the inner loop in Step~1 primarily follows from Lemma~\ref{lem:2}(c).  
(ii) Since {\bf (A1)} already ensures the boundedness of the Hessian approximations (see \eqref{Li:Hessi}), no specific conditions are imposed on the Hessian sample size.
(iii) The trust-region subproblem in \eqref{eq:subproblem} is solved inexactly, ensuring that condition \eqref{sdcdf} is satisfied. Specifically, the step \( d_k \) achieves at least the reduction in the model provided by the Cauchy step \( d^C_k \), which is defined as  
\begin{equation*}
   d^C_k := \argmin \{ m_k(d) : d = -t \nabla f_{\mathcal{G}_k}(x_k),\,\,t > 0, \, \|d\| \leq \Delta_k \}.
\end{equation*}
(iv) As is customary in trust-region methods, the acceptance rule for \( d_k \) in Step~4 is based on the agreement between the function \( f \) at {\( x_k+d_{k} \)} and the model \( m_k \) evaluated at \( d_k \). 
\end{remark}

We now turn our attention to the iteration-complexity analysis of Algorithm~1. In this context, we classify the 
{iterations} of the algorithm into two distinct types:  
\vspace{0.2cm}
\begin{enumerate}
\item \textbf{Successful iterations} ($\mathcal{S}$): These occur when $\rho_{k} \geq \alpha$, resulting in an update $x_{k+1} = x_{k} + d_{k}$ and {a potential increase in the} trust region radius, $\Delta_{k+1} = \min\{2\Delta_{k}, \Delta_{\max}\}$.

\item \textbf{Unsuccessful iterations} ($\mathcal{U}$): These occur when $\rho_{k} < \alpha$, where the point remains unchanged, $x_{k+1} = x_{k}$, and the trust region radius is reduced, $\Delta_{k+1} = \frac{1}{2} \Delta_{k}$.
\end{enumerate}
The following index sets are required: for a given \( k \in \{0, 1, 2, \dots \} \), define  
\begin{equation}\label{indexSU}
  \mathcal{S}_k = \{0, 1, \dots, k\} \cap \mathcal{S}, \quad \mathcal{U}_k = \{0, 1, \dots, k\} \cap \mathcal{U}.  
\end{equation}
Additionally, define
\begin{equation}
T(\varepsilon_g) = \inf \left\{ k \in \mathbb{N} : \|\nabla f(x_k)\| \leq \varepsilon_g \right\}
   \label{eq:hitting}
\end{equation}
as the index of the first iteration {for} which 
{$x_{k}$ is an} \(\varepsilon_g\)-approximate stationary point. Our objective is to derive a finite upper bound for \( T(\varepsilon_g) \). By {assuming that \( T(\varepsilon_g) \geq 1 \)}, we have
\begin{equation}\label{eq:motivation}
   T(\varepsilon_g) = |\mathcal{S}_{T(\varepsilon_g)-1} \cup \mathcal{U}_{T(\varepsilon_g)-1}| \leq |\mathcal{S}_{T(\varepsilon_g)-1}| + |\mathcal{U}_{T(\varepsilon_g)-1}|. 
\end{equation}
{Thus, our analysis will focus on bounding from above $|\mathcal{S}_{T(\varepsilon_g)-1}|$ and $|\mathcal{U}_{T(\varepsilon_g)-1}|$}.

The next lemma shows that the sequence generated by the algorithm is well-defined and is contained in the level set $\cL_f(x_0)$. Moreover, it is proven that the  inner procedure in {Step~1} stops in a finite number of trials.  


\begin{lemma}\label{lem:3}
Suppose that {\bf (A1)--(A3)} hold and \( T(\varepsilon_g) \geq 1 \). 
Then, the sequence $\{x_k\}_{k=0}^{T(\varepsilon_g)}$ is well-defined and  is contained in $\cL_f(x_0)$. Moreover, 
the inner sequence  $\{j_k\}_{k=0}^{T(\varepsilon_g)-1}$ satisfies
\begin{equation}
 0\leq j_{k}\leq 1+\max\left\{\log_{\gamma}\left(10L_{g}D_{0}\varepsilon_{g}^{-1}\right),0\right\}:=j_{\max}.
\label{eq:3.7}
\end{equation}
\end{lemma}
\begin{proof} 
We proceed by induction. Clearly, \( x_0 \in \cL_f(x_0) \). We now verify that \eqref{eq:3.7} is satisfied for \( k=0 \).
{Suppose by contradiction that
\[
j_0 > 1 + \max \left\{ \log_{\gamma} \left(10L_{g}D_{0}\varepsilon_{g}^{-1}\right), 0 \right\}.
\]
Then we would have $j_{0}-1>\log_{\gamma}\left(10L_{g}D_{0}\varepsilon_{g}^{-1}\right)$, which by (\ref{def:h_k}) and $\Delta_{0}\leq \Delta_{\max}$, would imply 
\begin{equation*}
    h_{0}^{j_{0}-1}=\frac{\Delta_{0}}{\gamma^{j_{0}-1}\Delta_{\max}}\leq\dfrac{1}{\gamma^{j_{0}-1}}<\frac{\varepsilon_{g}}{10L_{g}D_{0}}.
\end{equation*}}
By Lemma~\ref{lem:2}(c), with \( x := x_0 \), {\( h = h^{j_0 - 1}_0 \) and \( \mathcal{G} := \mathcal{G}^{j_0 - 1}_0 \)}, it follows that the inequality in \eqref{graSubcond} would hold for \( j = j_0 - 1 \), contradicting the minimality of \( j_0 \). Therefore, \eqref{eq:3.7} holds for \( k = 0 \).
Since the inner procedure terminates after a finite number of steps, we conclude that \( x_1 \) is well-defined and belongs to \( \cL_f(x_0) \) (see Step~4 of Algorithm~1). Now, assuming \( x_k \in \cL_f(x_0) \) holds, a similar argument shows that \( j_k \) satisfies \eqref{eq:3.7}, which implies that \( x_{k+1} \) is well-defined and also belongs to \( \cL_f(x_0) \). 
\end{proof}

\noindent In view of (\ref{Li:Hessi}), let us now consider the following assumption on the sequence of matrices $\left\{B_{k}\right\}_{k\geq 0}$:
\begin{itemize}
    \item [\textbf{(A4)}] For all $k\geq 0$, $\|B_{k}\|\leq L_{g}$.
\end{itemize}

Next, we derive a sufficient condition to ensure that an iteration is successful.

\begin{lemma}\label{lem:4}
Suppose that {\bf (A1)--(A4)} hold and \( T(\varepsilon_g) \geq 1 \).  {Given $k\leq T(\varepsilon_g)-1$, if  
\begin{equation}\label{Es:delta}
    \Delta_k \leq \frac{(1-\alpha)\|\nabla f_{\mathcal{G}_k}(x_k)\|}{2\left[1+2\left(\frac{D_{0}}{\Delta_{\max}}\right)\right]L_{g}},
\end{equation}
}
 then $\rho_k\geq \alpha$ (that is, $k \in \mathcal{S}_{T(\varepsilon_g)-1}$).
\end{lemma}
\begin{proof} It follows from \eqref{ratio_generalset} and  \eqref{eq:subproblem} that 
    \begin{eqnarray*}
1-\rho_{k}&=&\frac{m(0)-m(d_{k})-f(x_k)+
f(x_k+d_{k})}{m(0)-m(d_{k})}\\
&=& \frac{
f(x_k+d_{k})-f(x_k)-\langle \nabla f_{\mathcal{G}_k}(x_k),d_k\rangle-\dfrac{1}{2}\langle B_{k}d_k,d_k\rangle}{m(0)-m(d_{k})}\\
&=& \frac{f(x_k+d_{k})-f(x_k){-\langle \nabla f(x_k),d_k\rangle}-\langle \nabla f_{\mathcal{G}_k}(x_k)-\nabla f(x_k),d_k\rangle-\dfrac{1}{2}\langle B_{k}d_k,d_k\rangle}{m(0)-m(d_{k})}
\end{eqnarray*}
From the last inequality, \eqref{eq:2.2}, the Cauchy-Schwartz inequality, Step~1 of the Algorithm~1 and  Lemma~\ref{lem:2}(a), we find
\begin{align*}
 1-\rho_{k}&\leq   \frac{ {L_g}\|d_k\|^2+ \|\nabla f_{\mathcal{G}_k}(x_k)-\nabla f(x_k)\|\|d_k\| }{m(0)-m(d_{k})}\\
 &\leq   \frac{ {L_g}\|d_k\|^2+ 2h_k^{j_k}L_g\D_0\|d_k\| }{m(0)-m(d_{k})},
\end{align*}
which, combined with the definition of $h_k^{j_k}$ in \eqref{def:h_k}, $\gamma>1$ and $\|d_k\|\leq \Delta_k$, yields{
\begin{align*}
 1-\rho_{k}
 &\leq   \frac{ \left[1+2\left(\frac{D_{0}}{\Delta_{\max}}\right)\right]{L_g}\Delta_k^2}{m(0)-m(d_{k})}.
\end{align*}
}
On the other hand, since {
\[
 \Delta_k \leq \frac{(1-\alpha)\|\nabla f_{\mathcal{G}_k}(x_k)\|}{2\left[1+2\left(\frac{D_{0}}{\Delta_{\max}}\right)\right]L_{g}}\leq \frac{\|\nabla f_{\mathcal{G}_k}(x_k)\|}{L_g},
\]
}
it follows from \eqref{Li:Hessi} and  \eqref{sdcdf}  that
\begin{align*}
 \frac{ 1}{m(0)-m(d_{k})}
 &\leq \frac{2}{\|\nabla f_{\mathcal{G}_k}(x_k)\|\min\left\{\Delta_k,\frac{\|\nabla f_{\mathcal{G}_k}(x_k)\|}{L_g}\right\} }= \frac{2}{\Delta_k\|\nabla f_{\mathcal{G}_k}(x_k)\|}.
\end{align*}
By combining the last inequalities, we obtain 
\begin{align*}
 1-\rho_{k}
 &\leq   \frac{ 2\left[1+2\left(\frac{D_{0}}{\Delta_{\max}}\right)\right]L_{g}}{\|\nabla f_{\mathcal{G}_k}(x_k)\|}\Delta_k.
\end{align*}
Therefore, the desired inequality follows now from \eqref{Es:delta}.
\end{proof}
The following lemma establishes a lower bound for the trust-region radius.

\begin{lemma}\label{lem:5}
Suppose that {\bf (A1)--(A4)} hold and \( T(\varepsilon_g) \geq 1 \).  Then,
\begin{equation}\label{Es:341}
 \Delta_k \geq \Delta_{\min}(\varepsilon_g), \quad \forall k\leq T(\varepsilon_g)-1,
\end{equation}
where
\begin{equation}\label{Es:3412}
 \Delta_{\min}(\varepsilon_g):=\min\left\{\Delta_0,\frac{(1-\alpha)\varepsilon_g}{5\left[1+2\left(\frac{D_{0}}{\Delta_{\max}}\right)\right]L_{g}}\right\}. 
\end{equation}
\end{lemma}
\begin{proof}
   Clearly \eqref{Es:341}) is true for  $k=0$.  Suppose that \eqref{Es:341}) holds for some $k\geq 0$, and let us prove that the inequality also holds for $k+1$. We consider two case:
   \\[2mm]
   \textbf{Case I:} $\Delta_{k}\leq\dfrac{2(1-\alpha)\varepsilon_{g}}{5\left[1+2\left(\frac{D_{0}}{\Delta_{\max}}\right)\right]L_{g}}$.
   \\[0.2cm]
    Since $\|\nabla f_{\mathcal{G}_k}(x_k)\|> {4\varepsilon_g}/5$, in this case, we have  
   \[
   \Delta_{k}\leq \frac{(1-\alpha)\|\nabla f_{\mathcal{G}_{k}}(x_{k})\|}{2\left[1+2\left(\frac{D_{0}}{\Delta_{\max}}\right)\right]L_{g}}.
   \]
   Therefore, it follows from Lemma~\ref{lem:4} that  $\rho_k\geq \alpha$, which in turn implies that 
   \[
   \Delta_{k+1}=\min\left\{2\Delta_{k},\Delta_{\max}\right\}\geq \Delta_k \geq  \Delta_{\min}(\varepsilon_g),
   \]
where the last inequality is due to the induction hypothesis. Thus,   \eqref{Es:341}) is true  for $k+1$.
\\[2mm]
\textbf{Case II} $\Delta_{k}>\dfrac{2(1-\alpha)\varepsilon_{g}}{5\left[1+2\left(\frac{D_{0}}{\Delta_{\max}}\right)\right]L_{g}}$. 
\\[0.2cm]
Since the trust region radius in Algorithm~1 satisfies  $\Delta_{k+1}\geq \frac12\Delta_{k}$, it follows that
   \[
   \Delta_{k+1}\geq\frac{1}{2}\Delta_{k} >  \dfrac{(1-\alpha)\varepsilon_{g}}{5\left[1+2\left(\frac{D_{0}}{\Delta_{\max}}\right)\right]L_{g}} \geq  \Delta_{\min}(\varepsilon_g) 
   \]
proving  \eqref{Es:341}) for $k+1$. 
\end{proof}
 Let us now consider the following additional assumption:
\begin{itemize}
\item [\textbf{(A5)}] The inital trust-region radius $\Delta_{0}>0$ is chosen such that 
\begin{equation}\label{Delta0}
 \Delta_{0}\geq\dfrac{(1-\alpha)\epsilon_{g}}{5\left[1+2\left(\frac{D_{0}}{\Delta_{\max}}\right)\right]L_{g}}.
\end{equation}
\end{itemize}

In the following two lemmas, we will derive upper bounds for
$|\mathcal{S}_{T(\varepsilon_g)-1}|$ and $  |\mathcal{U}_{T(\varepsilon_g)-1}|$.

\begin{lemma}
\label{lem:3.454}
Suppose that {\bf (A1)--(A5)} hold and \( T(\varepsilon_g) \geq 1 \). Then
\begin{equation}
|\mathcal{S}_{T(\varepsilon_g)-1}|\leq \dfrac{25\left[1+2\left(\frac{D_{0}}{\Delta_{\max}}\right)\right]L_{g}(f(x_{0})-f_{low})}{2\alpha(1-\alpha)}\epsilon_{g}^{-2}.
\label{eq:3.11}
\end{equation}
\end{lemma}
\begin{proof}
  Let {$k\in S_{T(\varepsilon_g)-1}$}, that is, 
$\rho_{k}\geq\alpha$.  Then, by \eqref{ratio_generalset}, \eqref{sdcdf}, \eqref{graSubcond},  \eqref{Li:Hessi}, Lemma~\ref{lem:5} and (\textbf{A5}), 
\begin{eqnarray}
f(x_{k})-f(x_{k+1})&\geq &\alpha\left[m(0)-m(d_{k})\right]\nonumber\\
&\geq &\frac{\alpha}{2}\|\nabla f_{\mathcal{G}_k}(x_k)\|\min\left\{\Delta_k,\frac{\|\nabla f_{\mathcal{G}_k}(x_k)\|}{\|B_{k}\|}\right\} \nonumber\\
&\geq &\frac{2\alpha\varepsilon_g}{5}\min\left\{\Delta_k,\frac{4\varepsilon_g}{5L_g}\right\} \nonumber\\
&\geq &\frac{2\alpha\varepsilon_g}{5}\min\left\{\Delta_{\min}(\varepsilon_{g}),\frac{4\varepsilon_{g}}{5L_{g}}\right\}\nonumber\\
&= & \dfrac{2\alpha\epsilon_{g}}{5}\dfrac{(1-\alpha)\varepsilon_{g}}{5\left[1+2\left(\frac{D_{0}}{\Delta_{\max}}\right)\right]L_{g}}\nonumber\\
& = &\dfrac{2\alpha(1-\alpha)}{25\left[1+2\left(\frac{D_{0}}{\Delta_{\max}}\right)\right]L_{g}}\varepsilon_{g}^{2}.
\label{eq:gg1}
\end{eqnarray}
Let $\mathcal{S}_{T(\varepsilon_g)-1}^{c}=\left\{0,1,\ldots,T(\varepsilon_g)-1\right\}\setminus \mathcal{S}_{T(\varepsilon_g)-1}$. Notice that, when $k\in\mathcal{S}_{T(\varepsilon_g)-1}^{c}$, we have the equality $f(x_{k+1})=f(x_{k})$. Thus, it follows from  \eqref{inf:f} {and (\ref{eq:gg1})} that
\begin{eqnarray*}
f(x_{0})-f_{low}&\geq & f(x_{0})-f(x_{T(\varepsilon_g)})=\sum_{k=0}^{T(\varepsilon_g)-1}f(x_{k})-f(x_{k+1})\\
&= & \sum_{k\in\mathcal{S}_{T(\varepsilon_g)-1}}f(x_{k})-f(x_{k+1})+\sum_{k\in\mathcal{S}_{T(\varepsilon_g)-1}^{c}}f(x_{k})-f(x_{k+1})\\
&=&\sum_{k\in\mathcal{S}_{T(\varepsilon_g)-1}}f(x_{k})-f(x_{k+1})\\
&\geq &|\mathcal{S}_{T(\varepsilon_g)-1}|\dfrac{2\alpha(1-\alpha)}{25\left[1+2\left(\frac{D_{0}}{\Delta_{\max}}\right)\right]L_{g}}\epsilon_{g}^{2},
\end{eqnarray*}
which implies that (\ref{eq:3.11}) is true.  
\end{proof}

\begin{lemma}
\label{lem:3.578}
Suppose that {\bf (A1)--(A5)} hold and \( T(\varepsilon_g) \geq 1 \). Then,
\begin{equation}
|\mathcal{U}_{T(\varepsilon_g)-1}|\leq\log_{2}\left(\frac{5\left[1+2\left(\frac{D_{0}}{\Delta_{\max}}\right)\right]L_{g} \Delta_{0}}{(1-\alpha)}\epsilon_{g}^{-1}\right)+|\mathcal{S}_{T(\varepsilon_g)-1}|.
\label{eq:3.13}
\end{equation}
\end{lemma}
\begin{proof}
By the update rules for $\Delta_{k}$ in Algorithm~1, we have
\begin{eqnarray*}
\Delta_{k+1}&=&\frac{1}{2}\Delta_{k},\,\,\text{if}\,\,k\in\mathcal{U}_{T(\varepsilon_g)-1},\\
\Delta_{k+1}&\leq & 2\Delta_{k},\,\,\,\text{if}\,\,k\in\mathcal{S}_{T(\varepsilon_g)-1}
\end{eqnarray*}
In addition, by Lemma~\ref{lem:5} we have
\begin{equation*}
\Delta_k \geq \Delta_{\min}(\varepsilon_g)\quad\text{for}\quad k=0,\ldots,T(\varepsilon_g),
\end{equation*}
where $\Delta_{\min}(\varepsilon_g)$ is as in \eqref{Es:3412}. Thus, defining  $\omega_{k}:=1/\Delta_{k}$, it follows that
\begin{eqnarray}
2\omega_{k}&=&\omega_{k+1},\quad\text{if}\,\,k\in\mathcal{U}_{T(\varepsilon_g)-1},\label{eq:3.14}\\
\frac{1}{2}\omega_{k}&\leq & \omega_{k+1},\quad\text{if}\,\,k\in\mathcal{S}_{T(\varepsilon_g)-1},
\label{eq:3.16}
\end{eqnarray}
and
\begin{equation}
\omega_{k}\leq(\Delta_{\min}(\varepsilon_g))^{-1}\quad\text{for}\quad k=0,\ldots,T(\varepsilon_g).
\label{eq:3.17}
\end{equation}
In view of (\ref{eq:3.14})-(\ref{eq:3.17}), we have
\begin{equation*}
2^{\left|\mathcal{U}_{T(\varepsilon_g)-1}\right|}\left(0.5\right)^{\left|\mathcal{S}_{T(\varepsilon_g)-1}\right|}\omega_{0}\leq\omega_{T(\varepsilon_g)}\leq(\Delta_{\min}(\varepsilon_g))^{-1}.
\end{equation*}
Then, taking the logarithm in both sides we get
\begin{equation}\label{eq:pr45} 
\left|\mathcal{U}_{T(\varepsilon_g)-1}\right|-|\mathcal{S}_{T(\varepsilon_g)-1}|\leq\log_{2}\left(\frac{(\Delta_{\min}(\varepsilon_g))^{-1}}{\omega_{0}}\right)=\log_{2}\left(\frac{\Delta_{0}}{\Delta_{\min}(\varepsilon_g)}\right).
\end{equation}
On the other hand, using  \eqref{Es:3412} and (\textbf{A5}), we obtain 
\[
\frac{\Delta_{0}}{\Delta_{\min}(\varepsilon_g)}=\frac{5\left[1+2\left(\frac{D_{0}}{\Delta_{\max}}\right)\right]L_{g}\Delta_{0}}{(1-\alpha)}\epsilon_{g}^{-1}.
\]
The inequality in (\ref{eq:3.13}) now follows by combining the last two inequalities. 
\end{proof}

Combining the preceding results, we derive the following worst-case complexity bound for the number of iterations needed by Algorithm~1 to achieve an \(\varepsilon_g\)-approximate stationary point.

\begin{theorem} \label{thm:3.1}
Suppose that {\bf (A1)--(A5)} hold, and let $T(\varepsilon_g)$ be as in  \eqref{eq:hitting}. Then,
\begin{equation}
T(\varepsilon_g)\leq \dfrac{25\left[1+2\left(\frac{D_{0}}{\Delta_{\max}}\right)\right]L_{g}(f(x_{0})-f_{low})}{\alpha(1-\alpha)}\epsilon_{g}^{-2}+\log_{2}\left(\frac{5\left[1+2\left(\frac{D_{0}}{\Delta_{\max}}\right)\right]L_{g}{\Delta_{0}}}{(1-\alpha)}\epsilon_{g}^{-1}\right) + 1,
\label{eq:3.23}
\end{equation}
\end{theorem}
\begin{proof}
If $T(\varepsilon_g)\leq 1$, then (\ref{eq:3.23}) is clearly true. Let us assume that $T(\varepsilon_g)\geq 2$. By (\ref{eq:motivation}),
\begin{equation*}
T(\varepsilon_g)\leq  \left|\mathcal{S}_{T(\varepsilon_g)-1}\right|+\left|\mathcal{U}_{T(\varepsilon_g)-1}\right|.
\end{equation*}
Then, (\ref{eq:3.23}) follows directly from Lemmas \ref{lem:3.454} and \ref{lem:3.578}.
\end{proof}

\section{Method for Computing Approximate Second-Order Critical Points}\label{sec:4}

In this section, we propose and analyze a variant of Algorithm~1 designed to find an $(\varepsilon_{g},\varepsilon_{H})$-approximate second-order critical point of $f$,  i.e., a point $x_{k}$ such that
 \begin{equation*}
     \|\nabla f(x_{k})\|\leq\varepsilon_{g}\quad\text{and}\quad \lambda_{\min}\left(\nabla^{2}f(x_{k})\right)\geq -\varepsilon_{H}.
 \end{equation*}
This method involves not only adjusting the gradient sample sizes but also accurately updating the Hessian approximations via subsampling techniques.
 We begin with a detailed description of the modified scheme.

\begin{mdframed}
\noindent\textbf{Algorithm 2:} \textbf{Second-Order sub-sampled-TR Method} 
\\[0.2cm]
\noindent\textbf{Step 0.} Given $x_0 \in \R^n$, $\varepsilon_g,\varepsilon_H  > 0$, $\gamma>1,$ $\alpha \in (0, 1)$, and $\Delta_{\max}\geq\Delta_{0}>0$, set $k:=0$.
\\[0.2cm]
\noindent\textbf{Step 1.} Let $j:=0$.
\\[0.2cm]
\textbf{Step 1.1.} Define 
\begin{equation} \label{def:h_k2345}
  h^j_{k,g}:=\dfrac{1}{\gamma^{j}}\left(\dfrac{\Delta_{k}}{\Delta_{\max}}\right)^{2}, 
\end{equation}
and choose $\G_{k}^j\subset \N$ such that $|\mathcal{G}_{k}^j|\geq\ceil{(1-h^j_{k,g})|\N|}$.
\\[0.2cm]
\textbf{Step 1.2.} Compute $\nabla f_{\mathcal{G}_k^j}(x_k)$. If 
\begin{equation}\label{graSubcond2354}
\|\nabla f_{\mathcal{G}_k^j}(x_k)\|> \frac{4\varepsilon_g}5,   
\end{equation}
set $j_k=j$ and $\G_k:=\mathcal{G}_k^{j_k}$, choose $\mathcal{H}_k\subset \N$ and compute $\nabla^2 f_{\mathcal{H}_k}(x_k)$,  and go to Step~2.
\\[0.2cm]
 \textbf{Step 1.3.} Define 
\begin{equation} \label{def:h_k59}
h^j_{k,H}:=\frac{\Delta_k}{\gamma^j\Delta_{\max}},
\end{equation}
and choose $\mathcal{H}_k^j\subset \N$ such that $|\mathcal{H}_k^j|\geq\ceil{(1-h^j_{k,H})|\N|}$.
\\[0.2cm]
\textbf{Step 1.4.} Compute $\nabla^2 f_{\mathcal{H}_k^j}(x_k)$. If 
\begin{equation}\label{graSubcond2245}
-\lambda_{\min}(\nabla^2 f_{\mathcal{H}_k^j}(x_k))>\frac{4\varepsilon_H}5,   
\end{equation}
set $j_k=j$, $\G_k:=\mathcal{G}_k^{j_k}$ and $\mathcal{H}_k:=\mathcal{H}_k^{j_k}$, and go to Step~2. 
Otherwise, set $j:=j+1$ and go  to Step 1.1. 
\\[0.2cm]
\noindent\textbf{Step 2} Compute an approximate solution $d_{k}$ of the trust-region subproblem \eqref{eq:subproblem} with $B_k=\nabla^2 f_{\mathcal{H}_k}(x_k)$ such that
\begin{equation}\label{sdcdf1}
   m_k(0)-m_k(d_k) \geq \max\left\{\frac12\|\nabla f_{\mathcal{G}_k}(x_k)\|\min\left\{\Delta_k,\frac{\|\nabla f_{\mathcal{G}_k}(x_k)\|}{\|\nabla^2 f_{\mathcal{H}_k}(x_k)\|}\right\},-\lambda_{\min}(\nabla^2 f_{\mathcal{H}_k}(x_k))\Delta_k^2\right\}.
    \end{equation}
\noindent\textbf{Step 3.} 
Compute \(\rho_k\) as in \eqref{ratio_generalset}, and update \(x_{k+1}\) and \(\Delta_{k+1}\) {as in} Step~4 of Algorithm~1. Set \(k := k+1\), and return to Step 1.
\end{mdframed}
\vspace{0.3cm}
\begin{remark}
In  Algorithm~2, where the goal is to find second-order critical points, greater care must be taken when updating the Hessian subsample size. Specifically, the size can be arbitrary if \eqref{graSubcond2354} is satisfied; otherwise, it must adhere to the rule defined in Step 1.3.
Moreover, the inexact criteria for solving the TR subproblem require a condition involving second-order information.
\end{remark}

As in Section~\ref{sec:3}, we will use the index sets \(\mathcal{S}_k\) and \(\mathcal{U}_k\), as defined in \eqref{indexSU}. Moreover, we define  
\begin{equation}
T(\varepsilon_g, \varepsilon_H) = \inf \left\{ k \in \mathbb{N} : \|\nabla f(x_k)\| \leq \varepsilon_g, \; \text{and} \; \lambda_{\min}(\nabla^2 f(x_k)) > -\varepsilon_H \right\}
\label{eq:defhitting}
\end{equation}
as the index of the first iteration for which $x_{k}$ is an \((\varepsilon_g, \varepsilon_H)\)-approximate second-order stationary point. Our goal is to establish a finite upper bound for \(T(\varepsilon_g, \varepsilon_H)\), where  
\begin{equation}
T(\varepsilon_g, \varepsilon_H) = |\mathcal{S}_{T(\varepsilon_g, \varepsilon_H) - 1} \cup \mathcal{U}_{T(\varepsilon_g, \varepsilon_H) - 1}| \leq |\mathcal{S}_{T(\varepsilon_g, \varepsilon_H) - 1}| + |\mathcal{U}_{T(\varepsilon_g, \varepsilon_H) - 1}|,
\label{eq:ggextra}
\end{equation}
when \(T(\varepsilon_g, \varepsilon_H) \geq 1\).

The next lemma shows, in particular, that \(\{x_k\}_{k=0}^{T(\varepsilon_g, \varepsilon_H)}\) is well-defined and that the inner procedure terminates in a finite number of trials.

\begin{lemma}\label{lem:3345}
Suppose that {\bf (A1)--(A3)} hold and \( T(\varepsilon_g,\varepsilon_H) \geq 1 \). 
Then, the sequence $\{x_k\}_{k=0}^{T(\varepsilon_g,\varepsilon_H)}$ is well-defined and  is contained in $\cL_f(x_0)$. Moreover, 
the inner sequence  $\{j_k\}_{k=0}^{T(\varepsilon_g,\varepsilon_H)-1}$ satisfies
\begin{equation}
0\leq j_{k}\leq 1+\max\left\{\log_{\gamma}\left(10L_{g}D_{0}\epsilon_{g}^{-1}\right),\log_{\gamma}\left(10L_{g}\epsilon_{H}^{-1}\right),0\right\}:=\bar{j}_{max}.
\label{eq:3.709}
\end{equation}
\end{lemma}
\begin{proof}
Using statements (c) and (d) of Lemma  \ref{lem:2}, the proof is similar to that of Lemma~\ref{lem:3}, and is therefore omitted. 
\end{proof}

To present the second-order iteration complexity bound for Algorithm~2, the following assumption is required:
\begin{itemize}
\item[\textbf{(A6)}] {The Hessian  $\nabla^2 f_{\mathcal{H}}$ is $L_H$-Lipschitz continuous  for every $\mathcal{H}\subset \mathcal{N}$, i.e.,
\begin{equation*}
\|\nabla^2 f_{\mathcal{H}}(y)-\nabla^2 f_{\mathcal{H}}(x)\|\leq L_H\|y-x\|,\quad \forall x,y\in\mathbb{R}^{n}.
\end{equation*}
}
\end{itemize}
It follows trivially from $\bf{A6}$ that
 \begin{equation}\label{eq:90e1}
f_{\mathcal{H}}(y)\leq f_{\mathcal{H}}( x)+\langle \nabla f_{\mathcal{H}}( x),y- x\rangle+\dfrac{1}{2}\langle \nabla^2 f_{\mathcal{H}}( x)(y- x),y- x\rangle+\dfrac{L_H}{6}\|y- x\|^{3}, \quad \forall \; \mathcal{H}\subset\mathcal{N},  x,y \in \R^n.
\end{equation}
Next, we derive a sufficient condition to guarantee the success of an iteration.

\begin{lemma}\label{lem:434}
Suppose that {\bf (A1)--(A3)} and {\bf (A6)}  hold and \( T(\varepsilon_g,\varepsilon_H) \geq 1 \).  Given $k\leq T(\varepsilon_g,\varepsilon_H)-1$, if 
\begin{equation}\label{Es:delta12}
    \Delta_k \leq {\bar\Delta(\varepsilon_{g},\varepsilon_{H})}:=\min\left\{\dfrac{2(1-\alpha)\epsilon_{g}}{5\left[1+2\left(\frac{\Delta_{0}}{\Delta_{\max}}\right)\right]L_{g}},\dfrac{4(1-\alpha)\epsilon_{H}}{5\left[\frac{L_{H}}{6}+2\left(\frac{D_{0}}{\Delta_{\max}}+1\right)\left(\frac{L_{g}}{\Delta_{\max}}\right)\right]}\right\},
\end{equation}
 then $\rho_k\geq \alpha$ (that is, $k \in \mathcal{S}_{T(\varepsilon_g,\varepsilon_H)-1}$).
\end{lemma}
\begin{proof} It follows from \eqref{ratio_generalset} and  \eqref{eq:subproblem} with $B_k=\nabla^2 f_{\mathcal{H}_k}(x_k)$ that 
    \begin{align}\nonumber
1-\rho_{k}&=\frac{m(0)-m(d_{k})-f(x_k)+
f(x_k+d_{k})}{m(0)-m(d_{k})}\\ \nonumber
&= \frac{
f(x_k+d_{k})-f(x_k)-\langle \nabla f_{\mathcal{G}_k}(x_k),d_k\rangle-\dfrac{1}{2}\langle \nabla^2f_{\mathcal{H}_k}(x_k)d_k,d_k\rangle}{m(0)-m(d_{k})}\\
&= \frac{f(x_k+d_{k})-f(x_k)-\langle \nabla f(x_k),d_k\rangle-\langle \nabla f_{\mathcal{G}_k}(x_k)-\nabla f(x_k),d_k\rangle-\dfrac{1}{2}\langle \nabla^2f_{\mathcal{H}_k}(x_k)d_k,d_k\rangle}{m(0)-m(d_{k})}. \label{eq:564}
\end{align}
We now consider two case:\\[2mm]
\textbf{Case I:} $\|\nabla f_{\mathcal{G}_k}(x_k)\|> 4\varepsilon_g/5$.
\\[0.2cm]
From \eqref{eq:564}, \eqref{eq:2.2}, the Cauchy-Schwartz inequality and Lemma~\ref{lem:2}(a), we find
\begin{align}
 1-\rho_{k}&\leq   \frac{ {L_g}\|d_k\|^2+ \|\nabla f_{\mathcal{G}_k}(x_k)-\nabla f(x_k)\|\|d_k\| }{m(0)-m(d_{k})}\nonumber\\
 &\leq   \frac{ {L_g}\|d_k\|^2+ 2h_{k,g}^{j_k}L_g\D_0\|d_k\| }{m(0)-m(d_{k})}.
 \label{eq:gg2}
\end{align}
In view of \eqref{def:h_k2345}, $\Delta_{k}\leq\Delta_{\max}$ and $\gamma>1$, we have  
\begin{equation}
h_{k,g}^{j_{k}}=\dfrac{1}{\gamma^{j_{k}}}\left(\frac{\Delta_{k}}{\Delta_{\max}}\right)^{2}\leq\dfrac{\Delta_{k}}{\Delta_{\max}}.
\label{eq:gg3}
\end{equation}
Then, combining (\ref{eq:gg2}), (\ref{eq:gg3}) and $\|d_{k}\|\leq\Delta_{k}$, it follows that
\begin{align}\label{er:901}
 1-\rho_{k}
 &\leq   \frac{\left[1+2\left(\frac{D_{0}}{\Delta_{\max}}\right)\right]L_{g}\Delta_{k}^{2}}{m(0)-m(d_{k})}.
\end{align}
On the other hand, since 
\[
\Delta_k \leq \bar\Delta(\varepsilon_{g},\varepsilon_{H})\leq \frac{4\varepsilon_g}{5L_g},
\]
it follows from \eqref{Li:Hessi} and  \eqref{sdcdf1}  that
\begin{align*}
 \frac{ 1}{m(0)-m(d_{k})}
 \leq \frac{2}{\|\nabla f_{\mathcal{G}_k}(x_k)\|\min\left\{\Delta_k,\frac{\|\nabla f_{\mathcal{G}_k}(x_k)\|}{L_g}\right\} }
  \leq \frac{5}{2\varepsilon_g\min\left\{\Delta_k,\frac{4\varepsilon_g}{5L_g}\right\} }
 \leq \frac{5}{2\varepsilon_g\Delta_k}.
\end{align*}
By combining the last  inequality with \eqref{er:901}, 
\begin{align*}
 1-\rho_{k}
 &\leq   \frac{ 5\left[1+2\left(\frac{D_{0}}{\Delta_{\max}}\right)\right]L_{g}\Delta_{k}}{2\varepsilon_g}.
\end{align*}
Therefore, the desired inequality follows now from \eqref{Es:delta12}.
\\[2mm]
\textbf{Case II:} $-\lambda_{\min}(\nabla^2 f_{\mathcal{H}_k}(x_k))>{4\varepsilon_H}/5$.
\\[0.2cm]
From \eqref{eq:564} and {$\|d_{k}\|\leq\Delta_{k}$, we obtain 
\begin{align}
1-\rho_{k}&\leq \frac{f(x_k+d_{k})-f(x_k)-\langle \nabla f(x_k),d_k\rangle-\dfrac{1}{2}\langle \nabla^2f(x_k)d_k,d_k\rangle}{m(0)-m(d_{k})}\nonumber \\
&+\frac{\| \nabla f_{\mathcal{G}_k}(x_k)-\nabla f(x_k)\|\|d_k\|+ \dfrac{1}{2}\| \nabla^2f_{\mathcal{H}_k}(x_k)- \nabla^2f(x_k))\| \|d_k\|^2}{m(0)-m(d_{k})}\nonumber\\
&\leq\frac{\frac{L_H\|d_k\|^3}{6}+ 2h_{k,g}^{j_{k}}L_g\D_0\|d_k\|+ 2h_{k,H}^{j_k}L_g \|d_k\|^2}{m(0)-m(d_{k})}\nonumber\\
&\leq \frac{\frac{L_H\Delta_{k}^3}{6}+ 2h_{k,g}^{j_{k}}L_g\D_0\Delta_{k}+ 2h_{k,H}^{j_k}L_g \Delta_{k}^2}{m(0)-m(d_{k})}
\label{eq:gg4}
\end{align}
where the second inequality is due to  \eqref{eq:90e1}, Lemma~\ref{lem:2}(a) and  \eqref{er:908}.  {In view of  \eqref{def:h_k2345}, \eqref{def:h_k59}, $\Delta_{k}\leq\Delta_{\max}$ and $\gamma>1$, we have
\begin{equation}
    h_{k,g}^{j_{k}}\leq\left(\dfrac{\Delta_{k}}{\Delta_{\max}}\right)^{2}\quad\text{and}\quad h_{k,H}^{j_{k}}\leq\dfrac{\Delta_{k}}{\Delta_{\max}}.
    \label{eq:gg5}
\end{equation}
Then, combining (\ref{eq:gg4}) and (\ref{eq:gg5}), it follows that 
\begin{align*}
1-\rho_{k}
&\leq\dfrac{\left[\frac{L_{H}}{6}+2\left(\frac{D_{0}}{\Delta_{\max}}+1\right)\left(\frac{L_{g}}{\Delta_{\max}}\right)\right]\Delta_{k}^{3}}{m(0)-m(d_{k})}.
\end{align*}}
On the other hand, it follows from  \eqref{sdcdf1}  and the fact that $-\lambda_{\min}(\nabla^2 f_{\mathcal{H}_k}(x_k))>{4\varepsilon_H}/5$ that
\begin{align*}
 \frac{ 1}{m(0)-m(d_{k})}
 \leq \frac{1}{-\lambda_{\min}(\nabla^2 f_{\mathcal{H}_k}(x_k))\Delta_k^2 }
  \leq \frac{5}{4\varepsilon_H\Delta_k^2 }.
\end{align*}
By combining the last two inequalities, we find that
\begin{align*}
1-\rho_{k}\leq \dfrac{5\left[\frac{L_{H}}{6}+2\left(\frac{D_{0}}{\Delta_{\max}}+1\right)\left(\frac{L_{g}}{\Delta_{\max}}\right)\right]\Delta_{k}}{4\epsilon_{H}}.
\end{align*}}
Therefore, the desired inequality follows now from \eqref{Es:delta12}.
\end{proof}
The following lemma establishes a lower bound for the trust-region radius.

\begin{lemma}\label{lem:51}
Suppose that {\bf (A1)--(A3)} and {\bf (A6)} hold and \( T(\varepsilon_g,\varepsilon_H) \geq 1 \).  Then,
\begin{equation}\label{Es:3411}
 \Delta_k \geq \Delta_{\min}(\varepsilon_g,\varepsilon_H):= {\min}\{\Delta_0, \bar \Delta(\epsilon_{g},\epsilon_{H})/2\}, \quad \forall k\leq T(\varepsilon_g,\varepsilon_H)-1,
\end{equation}
where $\bar \Delta(\epsilon_{g},\epsilon_{H})$ is defined in \eqref{Es:delta12}.
\end{lemma}
\begin{proof}
   Clearly, \eqref{Es:3411}) is true for  $k=0$.  Suppose that \eqref{Es:3411}) holds for some $k\geq 0$, and let us prove that the inequality also holds for $k+1$. We consider two case:
   \\[2mm]
   \textbf{Case I:} $\Delta_{k} \leq \bar \Delta(\varepsilon_g,\varepsilon_H)$.
   \\[0.2cm]
   In this case,  it follows from Lemma~\ref{lem:434} that  $\rho_k\geq \alpha$, which in turn implies that 
   \[
   \Delta_{k+1}=\min\left\{2\Delta_{k},\Delta_{\max}\right\}\geq 2\Delta_{k}\geq \Delta_{k} \geq \Delta_{\min}(\varepsilon_g,\varepsilon_H),
   \]
   where the last inequality is due to the induction hypothesis. Thus,   \eqref{Es:3411}) is true  for $k+1$.
\\[2mm]
\textbf{Case II:} $\Delta_{k} \geq  \bar\Delta(\varepsilon_g,\varepsilon_H)$. 
\\[0.2cm]
Since the {trust-region} radius in Algorithm~1 satisfies  $\Delta_{k+1}\geq \frac12\Delta_{k}$, it follows that
   \[
   \Delta_{k+1}\geq \frac12\Delta_{k} \geq  \frac{\bar\Delta(\epsilon_{g},\epsilon_{H})}2\geq  \Delta_{\min}(\varepsilon_g, \varepsilon_H),
   \]
proving  \eqref{Es:3411}) for $k+1$. 
\end{proof}

In the next two lemmas, we establish upper bounds for \( |\mathcal{S}_{T(\varepsilon_g, \varepsilon_H)-1}| \) and \( |\mathcal{U}_{T(\varepsilon_g, \varepsilon_H)-1}| \).
\begin{lemma}
\label{lem:3.4541}
Suppose that {\bf (A1)--(A3)} and {\bf (A5)} hold and \( T(\varepsilon_g,\varepsilon_H) \geq 1 \).  Then
\begin{equation}
|\mathcal{S}_{T(\varepsilon_g,\varepsilon_H)-1}|\leq \frac{5}{2\alpha}(f(x_{0})-f_{low})\max\left\{\epsilon_{g}^{-1}\Delta_{\min}(\epsilon_{g},\epsilon_{H})^{-1},\epsilon_{H}^{-1}\Delta_{\min}(\epsilon_{g},\epsilon_{H})^{-2}\right\},
\label{eq:3.111}
\end{equation}
where $\Delta_{\min}(\epsilon_{g},\epsilon_{H})$ is defined in (\ref{Es:3411}).
\end{lemma}
\begin{proof}
  Let $k\in S_{T(\varepsilon_g,\varepsilon_H)-1}$, that is, 
$\rho_{k}\geq\alpha$.  Since $k\leq T(\epsilon_{g},\epsilon_{H})$ we have two possibilities:\\[2mm]
{\bf Case I:} $\|\nabla f_{\mathcal{G}_k}(x_k)\|> 4\varepsilon_g/5$.
\\[0.2cm]
In this case, by \eqref{ratio_generalset}, \eqref{sdcdf1},   \eqref{Li:Hessi} and Lemma~\ref{lem:51}, we have
\begin{eqnarray}
f(x_{k})-f(x_{k+1})&\geq &\alpha\left[m(0)-m(d_{k})\right]\nonumber\\
&\geq &\frac\alpha2\|\nabla f_{\mathcal{G}_k}(x_k)\|\min\left\{\Delta_k,\frac{\|\nabla f_{\mathcal{G}_k}(x_k)\|}{\|\nabla^2 f_{\mathcal{H}_k}(x_k)\|}\right\} \nonumber\\
&\geq &\frac{2\alpha\varepsilon_g}{5}\min\left\{\Delta_k,\frac{4\varepsilon_g}{5L_g}\right\}\nonumber \\
&\geq &{\frac{2\alpha\varepsilon_g}{5}\Delta_{\min}(\epsilon_{g},\epsilon_{H})}
\label{eq:789}
\end{eqnarray}
{\textbf{Case II:}} $-\lambda_{\min}(\nabla^2 f_{\mathcal{H}_k}(x_k))>{4\varepsilon_H}/5$. 
\\[0.2cm]
In this case, by \eqref{ratio_generalset}, \eqref{sdcdf1}, and Lemma~\ref{lem:51}, we have
\begin{eqnarray}
f(x_{k})-f(x_{k+1})&\geq &\alpha\left[m(0)-m(d_{k})\right]\nonumber\\
&\geq &-\alpha\lambda_{\min}(\nabla^2 f_{\mathcal{H}_k}(x_k))\Delta_k^2\nonumber\\
&\geq &\frac{4\alpha\varepsilon_H\Delta_k^2}{5}\nonumber\\
&\geq &{\frac{4\alpha\epsilon_{H}}{5}\Delta_{\min}(\epsilon_{g},\epsilon_{H})^{2}.}
\label{eq:gg6}
\end{eqnarray}
Thus, in view of (\ref{eq:789}) and (\ref{eq:gg6}), we conclude that
\begin{equation}
f(x_{k})-f(x_{k+1})\geq  \frac{2\alpha}{5}\min\left\{\epsilon_{g}\Delta_{\min}(\epsilon_{g},\epsilon_{H}),\epsilon_{H}\Delta_{\min}(\epsilon_{g},\epsilon_{H})^{2}\right\} \quad \forall k\in\mathcal{S}_{T(\varepsilon_g,\varepsilon_H)-1}.
\label{eq:gg7}
\end{equation}
Let $\mathcal{S}_{T(\varepsilon_g,\varepsilon_H)-1}^{c}=\left\{0,1,\ldots,T(\varepsilon_g,\varepsilon_H)-1\right\}\setminus \mathcal{S}_{T(\varepsilon_g,\varepsilon_H)-1}$. Notice that, when $k\in\mathcal{S}_{T(\varepsilon_g,\varepsilon_H)-1}^{c}$, then $f(x_{k+1})=f(x_{k})$. Thus, it follows from  \eqref{inf:f} and (\ref{eq:gg7}) that
\begin{eqnarray*}
f(x_{0})-f_{low}&\geq & f(x_{0})-f(x_{T(\varepsilon_g,\varepsilon_H)})=\sum_{k=0}^{T(\varepsilon_g,\varepsilon_H)-1}f(x_{k})-f(x_{k+1})\\
&= & \sum_{k\in\mathcal{S}_{T(\varepsilon_g,\varepsilon_H)-1}}f(x_{k})-f(x_{k+1})+\sum_{k\in\mathcal{S}_{T(\varepsilon_g,\varepsilon_H)-1}^{c}}f(x_{k})-f(x_{k+1})\\
&=&\sum_{k\in\mathcal{S}_{T(\varepsilon_g,\varepsilon_H)-1}}f(x_{k})-f(x_{k+1})\\
&\geq &|\mathcal{S}_{T(\varepsilon_g,\varepsilon_H)-1}|{\frac{2\alpha}{5}\min\left\{\epsilon_{g}\Delta_{\min}(\epsilon_{g},\epsilon_{H}),\epsilon_{H}\Delta_{\min}(\epsilon_{g},\epsilon_{H})^{2}\right\},}
\end{eqnarray*}
which implies that (\ref{eq:3.111}) is true.  
\end{proof}


\begin{lemma}
\label{lem:3.5781}
Suppose that {\bf (A1)--(A3)} and {\bf (A6)} hold  and \( T(\varepsilon_g,\varepsilon_H) \geq 1 \).  Then, 
\begin{equation}
|\mathcal{U}_{T(\varepsilon_g,\varepsilon_H)-1}|\leq{\log_{2}\left(\frac{\Delta_{\max}}{\Delta_{\min}(\epsilon_{g},\epsilon_{H})}\right)}+|\mathcal{S}_{T(\varepsilon_g,\varepsilon_H)-1}|,
\label{eq:3.131}
\end{equation}
where $\Delta_{\min}(\epsilon_{g},\epsilon_{H})$ is defined in (\ref{Es:3411}).
\end{lemma}
\begin{proof}
Using a similar argument as in Lemma~\ref{lem:3.578} (see \eqref{eq:pr45}), we obtain
    \begin{equation*}
\left|\mathcal{U}_{T(\varepsilon_g)-1}\right|-|\mathcal{S}_{T(\varepsilon_g)-1}|\leq\log_{2}\left(\frac{\Delta_{0}}{\Delta_{\min}(\varepsilon_g,\varepsilon_H)}\right).
\end{equation*}
Then, the conclusion follows from the choice $\Delta_{\min}(\varepsilon_g,\varepsilon_H) \leq \Delta_{0} \leq \Delta_{\max}$.
\end{proof}


Finally, combining the two previous lemmas with (\ref{eq:ggextra}), we derive the following iteration-complexity bound.

\begin{theorem} \label{thm:3.11}
Suppose that {\bf (A1)--(A3)} and {\bf (A6)}  hold,  and let $T(\varepsilon_g,\varepsilon_H)$ be as in  \eqref{eq:defhitting}. Then,
\begin{align*}
T(\varepsilon_g,\varepsilon_H)&\leq \frac{5}{\alpha}(f(x_{0})-f_{low})\max\left\{\epsilon_{g}^{-1}\Delta_{\min}(\epsilon_{g},\epsilon_{H})^{-1},\epsilon_{H}^{-1}\Delta_{\min}(\epsilon_{g},\epsilon_{H})^{-2}\right\}\\
&+\log_{2}\left(\frac{\Delta_{\max}}{\Delta_{\min}(\epsilon_{g},\epsilon_{H})}\right)+1
\label{eq:3.231}
\end{align*}
where $\Delta_{\min}(\epsilon_{g},\epsilon_{H})$ is defined in (\ref{Es:3411}).
\end{theorem}

If additionally {\bf (A5)} holds, in view of   
 of (\ref{Es:3411}) and (\ref{Es:delta12}) , we have
\begin{eqnarray*}
    \Delta_{\min}(\epsilon_{g},\epsilon_{H})&=&\min\left\{\dfrac{(1-\alpha)\epsilon_{g}}{5\left[1+2\left(\frac{D_{0}}{\Delta_{\max}}\right)\right]L_{g}},\dfrac{2(1-\alpha)\epsilon_{H}}{5\left[\frac{L_{H}}{6}+2\left(\frac{D_{0}}{\Delta_{\max}}+1\right)\left(\frac{L_{g}}{\Delta_{\max}}\right)\right]}\right\}\\
    &=&\min\left\{\dfrac{(1-\alpha)\epsilon_{g}}{5\left[1+2\left(\frac{D_{0}}{\Delta_{\max}}\right)\right]L_{g}},\dfrac{2(1-\alpha)\epsilon_{H}}{\left[\frac{1}{6}+2\left(\frac{D_{0}}{\Delta_{\max}}+1\right)\left(\frac{L_{g}L_{H}^{-1}}{\Delta_{\max}}\right)\right]L_{H}}\right\}\\
    &=&\mathcal{O}\left(\min\left\{\frac{\epsilon_{g}}{L_{g}},\frac{\epsilon_{H}}{L_{H}}\right\}\right),
\end{eqnarray*}
and so
\begin{equation*}
    \Delta_{\min}(\epsilon_{g},\epsilon_{H})^{-1}=\mathcal{O}\left(\max\left\{L_{g}\epsilon_{g}^{-1},L_{H}\epsilon_{H}^{-1}\right\}\right)
\end{equation*}
Then, it follows from Theorem \ref{thm:3.11} that Algorithm 2 takes no more than 
\begin{equation}
\mathcal{O}\left((f(x_{0})-f_{low})\max\left\{L_{g}\epsilon_{g}^{-2},L_{H}\epsilon_{g}^{-1}\epsilon_{H}^{-1},L_{g}^{2}\epsilon_{g}^{-2}\epsilon_{H}^{-1},L_{H}^{2}\epsilon_{H}^{-3}\right\}\right)
\label{eq:gg8}
\end{equation}
iterations to find an $(\epsilon_{g},\epsilon_{H})$-approximate second-order critical point of $f$. Without loss of generality, we can assume that {$L_{g},L_{H}\geq 1$}. Then, taking $\epsilon_{g}$ and $\epsilon_{H}$ such that {$0<\epsilon_{H}\leq\epsilon_{g}<1$}, we get
\begin{equation*}
    L_{g}\epsilon_{g}^{-1}\leq L_{g}^{2}\epsilon_{g}^{-2}\epsilon_{H}^{-1}\quad\text{and}\quad L_{H}\epsilon_{g}^{-1}\epsilon_{H}^{-1}\leq L_{H}^{2}\epsilon_{H}^{-3}.
\end{equation*}
In this case, the complexity bound (\ref{eq:gg8}) reduces to 
\begin{equation*}
    \mathcal{O}\left((f(x_{0})-f_{low})\max\left\{L_{g}^{2}\epsilon_{g}^{-2}\epsilon_{H}^{-1},L_{H}^{2}\epsilon_{H}^{-3}\right\}\right).
\end{equation*}


\section{Illustrative Numerical Results}
\label{NumSec}

In this section, we report preliminary numerical results comparing an implementation of Algorithm~1 (referred to as \texttt{STR}) against an implementation of the standard trust-region method (referred to as \texttt{TR}). Both methods were applied to minimize the function
\begin{equation}
\min_{x \in \mathbb{R}^d} f(x) \equiv \frac{1}{d} \sum_{i=1}^d f_i(x), \quad \text{where} \quad f_i(x) = \left(d - \sum_{j=1}^f \cos(x_j) + i \bigl(1 - \cos(x_i)\bigr) - \sin(x_i)\right)^{2},
\label{eq:test_function}
\end{equation}
starting from the initial point \( x_0 = (1, \ldots, 1) \in \mathbb{R}^d \) and using the stopping criterion
\begin{equation}
\|\nabla f(x_k)\|_2 \leq \varepsilon_g=10^{-5}.
\label{eq:stopping_criterion}
\end{equation}
It is worth pointing out that the evaluation of the full gradient $\nabla f(x_k)$ in \texttt{STR} was done only to ensure a fair stopping criterion with the \texttt{TR}. These evaluations were not taken into account in the performance measure defined below. The other initialization parameters for the algorithms were set as $\Delta_0 = 1$, $\Delta_{\max} = 50$, $\gamma = 1.1$, and $\alpha = 10^{-4}$.
 In \texttt{TR}, the Hessian approximations $B_{k}$ were computed using the safeguarded BFGS formula:
 \begin{equation}
     B_{k}=\left\{\begin{array}{ll} I,&\text{if $k=0$,}\\
                                   B_{k-1}&\text{if $k\geq 1$ and $s_{k-1}^{T}y_{k-1}\leq 0$,}\\
                                   B_{k-1}+\frac{y_{k-1}y_{k-1}^{T}}{s_{k-1}^{T}y_{k-1}}-\dfrac{B_{k-1}s_{k-1}s_{k-1}^{T}B_{k-1}}{s_{k-1}^{T}B_{k-1}s_{k-1}},&\text{if $k\geq 1$ and $s_{k-1}^{T}y_{k-1}>0$,}
     \end{array}
     \right.
     \label{eq:bfgs}
 \end{equation}
 where $s_{k-1}=x_{k}-x_{k-1}$ and $y_{k-1}=\nabla f(x_{k})-\nabla f(x_{k-1})$. On the other hand, in the \texttt{STR} method, the Hessian approximations \( B_k \) were also computed using~(\ref{eq:bfgs}), but with vector \( y_k \) replaced by \( \hat{y}_k = \nabla f_{\mathcal{G}_k}(x_k) - \nabla f_{\mathcal{G}_{k-1}} (x_{k-1}) \). Furthermore, in each implementation, the trust-region subproblems are approximately solved using the Dogleg method. 

In Algorithm 1, one can choose any subset \(\mathcal{G}_k^j \subset \mathcal{N}\) in Step 1.1, provided it has the prescribed cardinality. At the \(k\)th iteration of the STR method, we consider the following strategy. We Set $|\mathcal{G}_{k}^j|=\ceil{(1-h^j_{k,g})|\N|}$. If \(k=0\) or \(k-1 \in \mathcal{S}\), we first identify an ordering \((i_1^k, \ldots, i_n^k)\) such that  
\[
f_{i_1^k}(x_k) \geq f_{i_2^k}(x_k) \geq \cdots \geq f_{i_n^k}(x_k),
\]
and then, for every \(j = 0, \ldots, j_k\), we choose  
\begin{equation}
\mathcal{G}_k^j = \left\{i_1^k, \ldots, i_{|\mathcal{G}_k^j|}^k\right\}.
\label{eq:choice_g}
\end{equation}  
On the other hand, if \(k-1 \in \mathcal{U}\), then the iterate does not change, i.e., \(x_{k+1} = x_k\). To reuse all previously computed gradients in this case, we select \(\mathcal{G}_k^j\) as in \eqref{eq:choice_g} but using the ordering from the previous iteration, that is,  
\[
i_\ell^k = i_\ell^{k-1} \quad \text{for } \ell = 1, \ldots, n.
\]
In this way, the total number of evaluations of \(\nabla f_i(\,\cdot\,)\) performed by \texttt{STR} from iterations \(0\) to \(T\) is equal to
\[
GE_T(\small\texttt{STR}\normalsize) := \sum_{k \in \mathcal{S}_T} |\mathcal{G}_k|,
\]
with \(1 \leq |\mathcal{G}_k| \leq d\), while the corresponding number of evaluations of \(\nabla f_i(\,\cdot\,)\) performed by \texttt{TR} is
\[
GE_T(\small\texttt{TR}\normalsize) := |\mathcal{S}_T| \times d.
\]
If each \(\nabla f_i(\cdot)\) is computed using reverse-mode Automatic Differentiation, it is reasonable to assume that its computational cost is approximately three times that of evaluating \(f_i(\cdot)\) once (see, e.g., \cite{GRIEWANK,AD1}). Since evaluating the full function $f(\,\cdot\,)$ requires computing all \(d\) component functions \(f_i(\,\cdot\,)\), we measure the \textit{total computational} cost up to iteration \(T\) in terms of equivalent evaluations of \(f_i(\,\cdot\,)\) as  
\begin{equation}
\text{Cost}_T = \left(FE_T \times d\right) + \left(3 \times GE_T\right),
\label{eq:cost}
\end{equation}
where \(FE_T\) is the number of full function evaluations \(f(\,\cdot\,)\) and \(GE_T\) is the number of component gradient evaluations \(\nabla f_i(\,\cdot\,)\) performed up to iteration \(T\). 

Table \ref{tab:1} reports the total cost incurred by \texttt{TR} and \texttt{STR} to compute an iterate \(x_k\) satisfying the stopping criterion~\eqref{eq:stopping_criterion}, for $f(\,\cdot\,)$ defined by (\ref{eq:test_function}) with various values of $d$. 
\small
\begin{table}[h!]
\centering
\begin{tabular}{|c|c|c|c|}
\hline
$d$ & $\text{Cost}(\texttt{TR})$ & $\text{Cost}(\texttt{STR})$ & Reduction \\ 
\hline
100 & 35,900  & 34,292 & 4\% \\\hline
500 & 194,500 & 117,097  & 39\%  \\\hline
1,000 & 626,000  & 419,053 & 33\% \\\hline
3,000 & 1,488,000 & 736,395 & 50\%  \\\hline
\end{tabular}
\caption{Comparison of iterations and cost between \texttt{TR} and \texttt{STR} across different problems.}
\label{tab:1}
\end{table}
\normalsize

As shown, \texttt{STR} can lead to a significant reduction in computational cost compared to \texttt{TR}, in terms of the cost metric defined in~(\ref{eq:cost}). This is because exact gradients are rarely required during the exacution of \texttt{STR}. Figure 1 illustrates how the sample size $|\mathcal{G}_{k}|$ evolved over the iterations for the problem with $d=3,000$ components. In this case, during several iterations, acceptable inexact gradients were computed using as few as 273 components, resulting in a 50\% reduction in the total computational cost.

\begin{figure}[h!]
    \centering
    \includegraphics[width=0.54\textwidth]{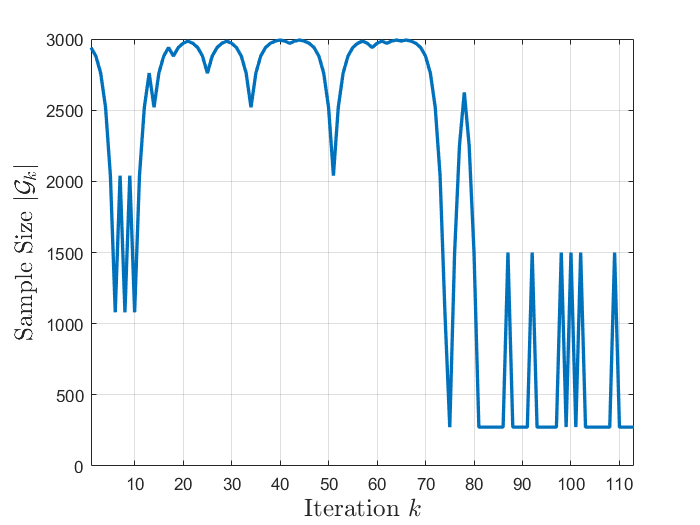} 
    \caption{Evolution of sample sizes for problem (\ref{eq:test_function}) with $d=3,000$.}
    \label{fig:sample}
\end{figure}

\section{Conclusion}
\label{concluding}

In this work, we introduced and analyzed sub-sampled trust-region methods  for solving finite-sum optimization problems. By employing random subsampling strategies to approximate the gradient and Hessian, these methods effectively reduce computational costs while maintaining theoretical guarantees. 
We established worst-case iteration complexity bounds for achieving approximate solutions. Specifically, we demonstrated that an \(\varepsilon_g\)-approximate first-order stationary point can be obtained in at most \(\mathcal{O}(\varepsilon_g^{-2})\) iterations and  an \((\varepsilon_g, \varepsilon_H)\)-approximate second-order stationary point is achievable within $\mathcal{O}(\max\{\varepsilon_{g}^{-2}\varepsilon_{H}^{-1},\varepsilon_{H}^{-3}\})$ iterations.
Numerical experiments confirmed the effectiveness of the proposed subsampling technique, highlighting its practical potential in solving finite-sum optimization problems.



\end{document}